\documentclass[a4paper,11pt]{article}

\usepackage{amsmath,amsthm,amssymb,mathrsfs,latexsym,amsfonts}
\usepackage{graphicx,psfrag,epsfig}
\usepackage{bbm}
\usepackage[english]{babel}
\usepackage[latin1]{inputenc}
\usepackage{pstricks}

\newtheorem{theorem}{Theorem}[section]
\newtheorem{lemma}[theorem]{Lemma}

\newtheorem{question}{Question}
\newtheorem{corollary}[theorem]{Corollary}
\newtheorem{definition}{Definition\rm}

\newcounter{paraga}[section]
\renewcommand{\theparaga}{{\bf\arabic{paraga}.}}
\newcommand{\paraga}{\medskip \addtocounter{paraga}{1} 
\noindent{\theparaga\ } }

\begin{document}

\def\MP{\,{<\hspace{-.5em}\cdot}\,}
\def\SP{\,{>\hspace{-.3em}\cdot}\,}
\def\PM{\,{\cdot\hspace{-.3em}<}\,}
\def\PS{\,{\cdot\hspace{-.3em}>}\,}
\def\EP{\,{=\hspace{-.2em}\cdot}\,}
\def\PP{\,{+\hspace{-.1em}\cdot}\,}
\def\PE{\,{\cdot\hspace{-.2em}=}\,}
\def\N{\mathbb N}
\def\C{\mathbb C}
\def\Q{\mathbb Q}
\def\R{\mathbb R}
\def\T{\mathbb T}
\def\A{\mathbb A}
\def\Z{\mathbb Z}
\def\demi{\frac{1}{2}}

\begin{titlepage}
\author{Pierre Berger~\footnote{pierre.berger@normalesup.org, CNRS and IMPA, Estrada Dona Castorina 110, Rio de Janeiro, Brasil, 22460-320} {} and Abed Bounemoura~\footnote{abed@impa.br, IMPA, Estrada Dona Castorina 110, Rio de Janeiro, Brasil, 22460-320}}
\title{\LARGE{\textbf{A geometrical proof of the persistence of normally hyperbolic submanifolds}}}
\end{titlepage}

\maketitle

\begin{abstract}
We present a simple, computation free and geometrical proof of the following classical result: for a diffeomorphism of a manifold, any compact submanifold which is invariant and normally hyperbolic persists under small perturbations of the diffeomorphism. The persistence of a Lipschitz invariant submanifold follows from an application of the Schauder fixed point theorem to a graph transform, while smoothness and uniqueness of the invariant submanifold are obtained through geometrical arguments. Moreover, our proof provides a new result on persistence and regularity of ``topologically" normally hyperbolic submanifolds, but without any uniqueness statement.
\end{abstract}
 
\section{Introduction}

\paraga Let $M$ be smooth manifold, $f : M \rightarrow M$ a $C^1$-diffeomorphism and $N\subseteq M$ a $C^1$-submanifold invariant by $f$. Roughly speaking, $f$ is \emph{normally hyperbolic} at $N$ if the tangent map $Tf$, restricted to the normal direction to $N$, is hyperbolic (it expands and contracts complementary directions) and if it dominates the restriction of $Tf$ to the tangent direction $TN$ (that is, expansion and contraction in the tangent direction, if any, are weaker than those in the normal direction). A precise definition will be given below.

The importance of invariant normally hyperbolic submanifolds, both in theoretical and practical aspects of dynamical systems, is well-known and it does not need to be emphasised. Let us just point out that quite recently, they have acquired a major role in establishing instability properties for Hamiltonian systems which are close to integrable, a problem which goes back to the question of the stability of the solar system.

\paraga It has been known for a long time that compact invariant normally hyperbolic submanifolds are \emph{persistent}, in the following sense: any diffeomorphism $g : M \rightarrow M$, sufficiently close to $f$ in the $C^1$-topology, leaves invariant and is normally hyperbolic at a submanifold $N_g$ $C^1$-close to $N$. Classical references for this result are \cite{HPS77} and \cite{Fen71}. 

In fact, normally hyperbolic submanifolds are persistent and \emph{uniformly locally maximal}: there exist neighbourhoods $U$ of $N$ in $M$ and $\mathcal{U}$ of $f$ in the space $\textrm{Diff}^1(M)$ of $C^1$-diffeomorphisms of $M$, such that for any $g\in\mathcal{U}$, $N_g=\bigcap_{k\in\Z}g^k(U)$ is a $C^1$-submanifold close to $N$, with $N_f=N$. The latter property implies uniqueness of the invariant submanifold.

The converse statement holds true: assuming $N$ is persistent and uniformly locally maximal, it was shown in \cite{Man78} that $N$ has to be normally hyperbolic.  

In the case where $N$ is a point, then it is a hyperbolic fixed point and the persistence follows trivially from the implicit function theorem. In the general case, however, such a direct approach is not possible and it is customary to deduce the persistence of compact normally hyperbolic submanifolds from the existence and persistence of the associated local stable (respectively unstable) manifolds, which are located in a neighbourhood of $N$ and are tangent to the sum of the contracting (respectively expanding) and tangent direction to $N$. The existence and persistence of stable and unstable manifolds have a long history, that we shall go through only very briefly. In the case of a fixed point, this was first proved by Hadamard (\cite{Had01}) who introduced the so-called ``graph transform" method, which relies on the contraction principle. Another proof, based on the implicit function theorem, was later given by Perron (\cite{Per28}), which was subsequently greatly simplified by Irwin (\cite{Irw70}). For normally hyperbolic submanifolds which are not reduced to a point, the result was proved independently in \cite{HPS77} and \cite{Fen71}, using the graph transform method. Moreover, in \cite{HPS77}, results on persistence were obtained not only for normally hyperbolic submanifolds but also in the more general context of normally hyperbolic laminations. This result was then clarified and further generalised in \cite{Ber10} and \cite{Ber11}, where not only laminations but also certain stratifications of normally hyperbolic laminations were shown to be persistent. As a final remark, let us point out that the graph transform method has been successfully applied for semi-flows in infinite dimension (\cite{BLZ98}), with a view towards applications to partial differential equations. 

\paraga The aim of this work is to give yet another proof of the persistence of compact normally hyperbolic submanifolds, a proof which we believe to be simpler. We will also use a graph transform, but at variance with all other proofs, we will not rely on the contraction principle. As a consequence, we will avoid any technical estimates that are usually required to show that the graph transform is indeed a contraction (on an appropriate Banach space). Also, we will be dispensed with giving the explicit, and usually rather cumbersome, expression of the graph transform acting on a suitable space of sections of a vector bundle. Instead, we will simply use the Schauder fixed point theorem to obtain the existence (but not the uniqueness) of the invariant submanifold. Such an invariant submanifold will be shown, at first, to be not more than Lipschitz regular. Then, to regain smoothness, we will use very simple geometrical properties of cone fields implied by the domination hypothesis. Compared to other proofs, and especially \cite{HPS77} where complicated techniques of ``Lipschitz jets" are used, our approach here is remarkably simple. 

Moreover, our proof yields a new result since the existence and regularity works for a wider class of submanifolds which we call ``topologically" normally hyperbolic, where basically we will retain a suitable domination property but the normal contraction and expansion will be replaced by topological analogues (see below for a precise definition). Under this weaker assumption no result of uniqueness has to be expected. In the classical normally hyperbolic case, using some other simple geometrical arguments (where the contraction property of the graph transform is hidden behind), the uniqueness will be established. 

The plan of the paper is the following. We state and explain the theorems for normally hyperbolic submanifolds in section~\ref{s2}, and for topologically normally submanifolds in section~\ref{s3}. The proofs of the results are given in section~\ref{s4}. 

\section{Normally hyperbolic submanifolds}\label{s2}

\paraga Let us now detail the setting that we shall use in the formulation of the theorem below. The Riemannian manifold $M$ is $m$-dimensional and smooth ({\it i.e.} $C^\infty$). The diffeomorphism $f : M \rightarrow M$ is (at least) of class $C^1$. The submanifold  $N\subseteq M$ is of class $C^1$, $n$-dimensional and \emph{closed}, that is without boundary, compact and connected. We suppose that  $f$ leaves $N$ \emph{invariant}: $f(N)=N$.

\begin{definition} The diffeomorphism $f$ is normally hyperbolic at $N$ if there exist a splitting of the tangent bundle of $M$ over $N$ into three $Tf$-invariant subbundles:
\[ TM_{|N}=E^s \oplus E^u \oplus TN  \]
and a constant $0<\lambda<1$ such that for all $x\in N$, with $||.||$ the operator norm induced by the Riemannian metric:
\begin{equation}\label{hyp}
||T_xf_{|E_x^s}|| < \lambda, \quad ||T_xf_{|E_x^u}^{-1}|| < \lambda, 
\end{equation} 
and
\begin{equation}\label{dom}
\begin{cases}
||T_xf_{|E_x^s}||\cdot||T_{f(x)}f_{|T_{f(x)}N}^{-1}|| < \lambda  \\
 ||T_xf_{|E_x^u}^{-1}||\cdot||T_{f^{-1}(x)}f_{|T_{f^{-1}(x)}N}|| < \lambda. 
\end{cases} 
\end{equation}
\end{definition}
One can check that the continuity of the splitting is then automatic. Condition~(\ref{hyp}) means that the normal behaviour of $Tf$ is hyperbolic while condition~(\ref{dom}) expresses the domination property with respect to the tangent behaviour of $Tf$ (in fact, condition~(\ref{dom}) can be expressed in many equivalent ways). If $E^u=\{0\}$ (respectively $E^s=\{0\}$), then $N$ is normally contracted (respectively normally expanded). Note that if the above conditions are satisfied not for $f$ but only for some iterate of $f$, then there exists another Riemannian metric (called \emph{adapted}) for which these two conditions hold true for $f$ (see \cite{Gou07} for the case of a general dominated splitting).

\paraga We endow the space of $C^1$-maps between $C^1$-manifolds with the compact-open topology (see \cite{Hir76}). This naturally defines a topology on the subset $\mathrm{Diff}^1(M)$ of $C^1$-diffeomorphisms of $M$. Moreover, two $C^1$-diffeomorphic submanifolds $N$ and $N'$ of $M$ are \emph{$C^1$-close} if there exists an embedding $i'$ of $N$ onto $N'$ which is $C^1$-close to the canonical inclusion $i : N \hookrightarrow M$. We are now ready to state the classical theorem.

\begin{theorem}\label{thmhps}
Let $f$ be a $C^1$-diffeomorphism which leaves invariant and is normally hyperbolic at a closed $C^1$-submanifold $N$. 
Then there exists a neighbourhood $\mathcal{U}$ of $f$ in $\mathrm{Diff}^1(M)$ such that any $g\in\mathcal{U}$ leaves invariant and is normally hyperbolic at a $C^1$-submanifold $N_g$, diffeomorphic and $C^1$-close to $N$. Moreover, $N_g$ is unique and uniformly locally maximal.
\end{theorem}

Let us point out that we will actually show the stable and unstable manifolds theorem: for any $g\in \mathcal{U}$, we will construct a local stable manifold $N^s_g$ (respectively a local unstable manifold $N^u_g$), $C^1$-close to $N^s_f$ (respectively $N^u_f$), the latter being the set of points whose forward (respectively backward) orbit lies in a small neighbourhood of $N$. Then the normally hyperbolic submanifold $N_g$ will be obtained as the (transverse) intersection between $N^s_g$ and $N^u_g$. 

\paraga Given any $r\geq 1$, we can replace condition~(\ref{dom}) by the stronger condition  
\begin{equation}\label{domr}
\begin{cases}
||T_xf_{|E_x^s}||\cdot||T_{f(x)}f_{|T_{f(x)}N}^{-1}||^k < \lambda \\
||T_xf_{|E_x^u}^{-1}||\cdot||T_{f^{-1}(x)}f_{|T_{f^{-1}(x)}N}||^k < \lambda 
\end{cases} \quad 1\leq k \leq r.
\end{equation}
If $f$ satisfies~(\ref{hyp}) and~(\ref{domr}), then it is \emph{$r$-normally hyperbolic} at $N$ (hence normally hyperbolic is just $1$-normally hyperbolic). Note that if condition~(\ref{hyp}) is satisfied, then it is sufficient to require condition~(\ref{domr}) only for $k=r$.

In this case, if $N$ and $f$ are $C^r$, using a trick from \cite{HPS77}, for $g$ in a $C^r$-neighbourhood of $f$, the invariant submanifold $N_g$ enjoys more regularity. Indeed, let $G_n(TM) \rightarrow M$ be the Grassmannian bundle with fibre at $x\in M$ the Grassmannian of $n$-planes of $T_x M$. The tangent map $Tf$ induces a canonical $C^{r-1}$-diffeomorphism $Gf$ of $G_n(TM)$. Moreover, the tangent bundle $TN$ can be considered as a closed $C^{r-1}$-submanifold of $G_n(TM)$ (although $TN$ is not compact as a submanifold of $TM$). As $f$ is $r$-normally hyperbolic at $N$, $Gf$ is $(r-1)$-normally hyperbolic at $TN$. By induction on $r\geq 1$, the existence and the uniqueness given by Theorem~\ref{thmhps} yields the following corollary.

\begin{corollary}\label{corhps}
For any integer $r\geq 1$, let $f$ be a $C^r$-diffeomorphism which leaves invariant and is $r$-normally hyperbolic at a closed $C^r$-submanifold $N$. Then there exists a neighbourhood $\mathcal{U}$ of $f$ in $\mathrm{Diff}^r(M)$ such that any $g\in\mathcal{U}$ leaves invariant and is $r$-normally hyperbolic at a $C^r$-submanifold $N_g$, diffeomorphic and $C^r$-close to $N$. Moreover, $N_g$ is unique and uniformly locally maximal.
\end{corollary}

Actually the assumption that $N$ is of class $C^r$ is automatic. In \cite{HPS77}, it is proved that if a $C^r$-diffeomorphism leaves invariant and is $r$-normally hyperbolic at a closed $C^1$-submanifold $N$, then $N$ is actually of class $C^r$. 

\section{Topologically normally hyperbolic submanifolds}\label{s3}

To prove the persistence of normally hyperbolic submanifolds, we will use a geometric model which can be satisfied without being normally hyperbolic. Therefore this enables us to weaken the assumptions on normal hyperbolicity to obtain a new result of persistence. The submanifolds satisfying such a geometric model will be called topologically normally hyperbolic.

\paraga We first explain the geometric model in a simple case. We consider a closed manifold $N$, that we identify to the submanifold $N\times \{0\}$ of 
\[ V=N\times \mathbb R^s\times \mathbb R^u,\] 
with $s,u\geq 0$. Given two compact, convex neighbourhoods $B^s$ and $B^u$ of $0$ in respectively $\mathbb R^s$ and $\mathbb R^u$, we define the following compact neighbourhood of $N$ in $V$:
\[ B = N\times B^s \times B^u, \]
and the subsets
\[ \partial^sB = N\times \partial B^s \times B^u, \quad \partial^u B = N\times B^s\times \partial B^u, \] 
where $\partial B^s$ (respectively $\partial B^u$) is the boundary of $B^s$ (respectively $B^u$). 

We assume that $f$ is a $C^1$-embedding of $B$ into $V$, that is a $C^1$-diffeomorphism of $B$ onto its image in $V$, which leaves $N$ invariant. Then we also define the following cones for $z\in V$:
\[C_z^u=\{v=v_1+v_2\in T_zV \; | \; v_1\in T_zN\times \mathbb R^s, v_2\in \mathbb R^u,\|v_2\|_z\leq\|v_1\|_z\},\]
\[C_z^s=\{v=v_1+v_2\in T_zV \; | \; v_1\in T_zN\times \mathbb R^u, v_2\in \mathbb R^s, \|v_2\|_z\leq\|v_1\|_z\},\]
with respect to a Riemannian metric on $V$.  

\begin{definition} 
Under the above assumptions, $f$ is topologically normally hyperbolic at $N$ if it satisfies the following conditions:
\begin{equation}\label{hyptop}
f(B) \cap \partial^sB=\emptyset, \quad B \cap f(\partial^uB)=\emptyset, \tag{$1'$} 
\end{equation} 
\begin{equation}\label{domtop}
T_zf(C_z^s)\subset \mathring C_{f(z)}^s\cup \{0\}, \quad C_{f(z)}^u\subset T_zf(\mathring C_z^u)\cup \{0\}, \tag{$2'$} 
\end{equation} 
for every $z\in B$.
\end{definition} 

Condition~(\ref{hyptop}) is equivalent to the requirement that $f(B)$ intersects the boundary of $B$ at most at $\partial^u B$, and that $f^{-1}(B)$ intersects the boundary of $B$ at most at $\partial^s B$. Moreover, by condition~(\ref{domtop}), we have the following transversality properties: for all $\xi_s\in \R^s$ and $\xi_u\in \R^u$, the image of the ``horizontal" $N\times\{\xi_s\}\times \R^u$ by $f$ intersects transversely each ``vertical" $N\times \R^s \times\{\xi_u\}$, and similarly, the image of $N\times \R^s \times\{\xi_u\}$ by $f^{-1}$ intersects transversely $N\times\{\xi_s\}\times \R^u$. The situation is depicted in figure~\ref{boite}. 
\begin{figure}[t] 
\centering
\def\JPicScale{0.7}
\input{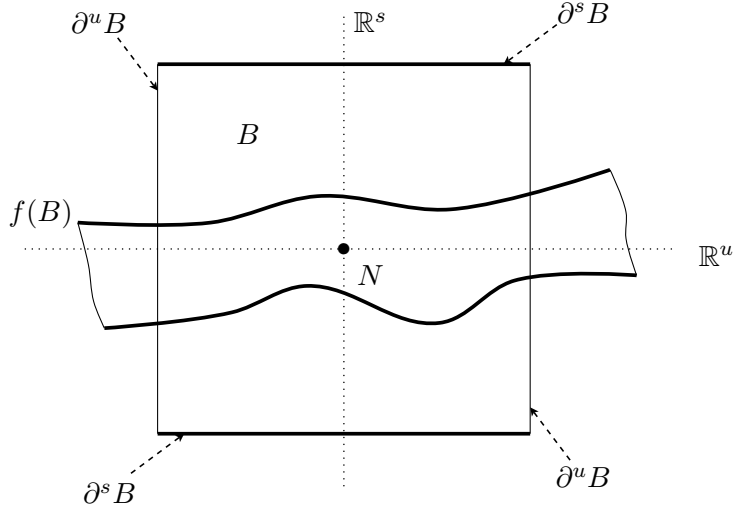}
\caption{A topologically normally hyperbolic submanifold} \label{boite}
\end{figure}  

When $u=0$, then $N$ is topologically normally contracted, and the characterisation is  simpler: $\partial B=\partial^sB$, the second half of condition~(\ref{hyptop}) is empty, the first half reads $f(B) \cap \partial B=\emptyset$ which can be seen to be equivalent to $f(B) \subseteq \mathring{B}$. A similar characterisation holds if $N$ is topologically normally expanded, that is when $s=0$.

Given a normally hyperbolic submanifold for which the stable and unstable bundles are trivial, we will construct a $C^1$-conjugacy of a neighbourhood of $N$ with such a geometric model. Basically, condition~(\ref{hyp}) will give condition~(\ref{hyptop}) and conditions~(\ref{hyp}) and~(\ref{dom}) will give condition~(\ref{domtop}). 

\paraga Now we shall generalise the concept of topological normal hyperbolicity, to include in particular normally hyperbolic submanifolds for which the stable and unstable bundle are not necessarily trivial. 

Let $V^s$ and $V^u$ be two vector bundles over a closed manifold $N$, whose fibres are of dimension $s$ and $u$, with $s,u\geq 0$. We denote by 
\[\pi:\; V=V^s\oplus V^u\rightarrow N\] 
the vector bundle whose fibre at $x\in N$ is the direct sum $V_x=V_x^s\oplus V_x^u$ of the fibres of $V^s$ and $V^u$ at $x\in N$. A \emph{horizontal distribution} for $\pi : V\rightarrow N$ is a smooth family of $n$-planes $(H_z)_{z\in V}$ such that $T_zV=H_z\oplus \ker T_z\pi$. Recall that a \emph{local trivialisation} of the vector bundle $\pi : V^s\oplus V^u\rightarrow N$ is an open set $W$ of $N$ and a diffeomorphism $\phi : \pi^{-1}(W)\rightarrow W\times \mathbb R^s\times \mathbb R^u$ such that its restriction to any fibre $\pi^{-1}(x)=V_x^s\oplus V_x^u$ is a linear automorphism onto $\{x\}\times \mathbb R^s\times \mathbb R^u$. A horizontal distribution $H$ is \emph{linear} if for any $z\in V_x^s\oplus V_x^u$, there exists a local trivialisation $\phi$ over a neighbourhood $W$ of $x$ such that $T_z\phi(H_z)= T_xW \times \{0\}$. Linear horizontal distribution exists on any vector bundle (the construction follows from the existence of a linear connection and parallel transport, see \cite{GHV73}).

Let $B^s\rightarrow N$ and $B^u\rightarrow N$ be two bundles whose fibres $B_x^s$ and $B_x^u$ at $x\in N$ are convex, compact neighbourhoods of $0$ in $V_x^s$ and $V_x^u$ respectively. Then we can define other bundles over $N$, namely $B= B^s \oplus B^u$, $\partial^sB = \partial B^s \oplus B^u$ and $\partial^u B = B^s\oplus \partial B^u$, where the fibre of $\partial B^s$ at $x \in N$ (respectively $\partial B^u$) is the boundary of $B_x^s$ (respectively $B_x^u$). Note that $B$, $\partial^sB$, and $\partial^u B$ are subbundles of $V$, and that $B$ is a compact neighbourhood of the graph of the zero section of $\pi : V\rightarrow N$.  

Let $f$ be a $C^1$-embedding of $B$ into $V$. Identifying $N$ to the graph of the zero section of $\pi : V\rightarrow N$, we assume that $f$ leaves $N$ invariant.

Given a linear horizontal distribution $H$ for $\pi : V\rightarrow N$ and a Riemannian metric on $V$, we define the following cones for $z\in V$:
\[C_z^u=\{v=v_1+v_2\in T_zV \; | \; v_1\in H_z\oplus V_{\pi(z)}^{s}, v_2\in V_{\pi(z)}^{u},\|v_2\|_z\leq\|v_1\|_z\},\]
\[C_z^s=\{v=v_1+v_2\in T_zV \; | \; v_1\in H_z\oplus V_{\pi(z)}^{u}, v_2\in V_{\pi(z)}^{s}, \|v_2\|_z\leq\|v_1\|_z\},\]
where we have identified the fibres $V_{\pi(z)}^{s}$ and $V_{\pi(z)}^{u}$ with their tangent spaces. Let us note that these cone fields are \emph{smooth} in the sense that $H_z, V_{\pi(z)}^{s}$ and $V_{\pi(z)}^{u}$ depend smoothly on $z\in B$.  

\begin{definition}\label{defgen} 
Under the above assumptions, $f$ is topologically normally hyperbolic at $N$ if it satisfies conditions~(\ref{hyptop}) and~(\ref{domtop}).
\end{definition}

This is clearly a generalisation of the simple geometric model, since the latter correspond to $V^s=N\times \R^s$, $V^u=N\times \R^s$ and there is a canonical linear horizontal distribution for the trivial vector bundle $\pi : V= N\times \R^s \times \R^u \rightarrow N$, given by $H_z=T_{\pi(z)}N \times \{0\}$ for $z\in V$.

\paraga Let us denote by $\mathrm{Emb}^1(B,V)$ the space of $C^1$-embeddings of $B$ into $V$, endowed with the $C^1$-topology. Here is the new result on persistence:

\begin{theorem}\label{thmhpstop}
Let $f\in\mathrm{Emb}^1(B,V)$ which leaves invariant and is topologically normally hyperbolic at a closed $C^1$-submanifold $N$. Then there exists a neighbourhood $\mathcal{B}$ of $f$ in $\mathrm{Emb}^1(B,V)$ such that any $g\in\mathcal{B}$ leaves invariant and is topologically normally hyperbolic at a $C^1$-submanifold $N_g$, diffeomorphic to $N$. 
\end{theorem}

Under the assumptions of this theorem, for $g\in \mathcal{B}$, we will see that $N_g$ satisfies the same geometric model. In particular, $N_g$ is included in $B$ and its tangent space is in $C^s\cap C^u$. Thus, if $N$ is topologically normally hyperbolic for a certain geometric model such that $B$ and $C^s\cap C^u$ can be taken ``arbitrarily small" (when viewed as neighbourhoods of respectively $N$ and $TN$), then $N_g$ is $C^1$-close to $N$ when $g$ is $C^1$-close to $f$. In general, this stronger assumption is not any more satisfied by $N_g$. 

The above result is analogous to Theorem~\ref{thmhps}, except that we do not obtain any uniqueness statement. One can only say that for all $g\in \mathcal{B}$, $g$ has at least one invariant submanifold $N_g$ contained in the maximal invariant subset of $B$. We will give below examples for which uniqueness fails, and as we already explained, the uniqueness property of Theorem~\ref{thmhps} is known to be characteristic of normal hyperbolicity.

\paraga Let us show that this generalisation is not free by giving some examples where Theorem~\ref{thmhps} fails, whereas Theorem~\ref{thmhpstop} applies.

The most simple example is when the submanifold has dimension zero, that is when it is a fixed point. Consider the map $f : \R^2 \rightarrow \R^2$ defined by
\[ f(x,y)=(x-x^3,2y), \quad (x,y)\in \R^2 \]
which is a diffeomorphism from a neighbourhood of the origin. The origin is a fixed point, which is not hyperbolic since the differential at this point has one eigenvalue equal to one. Hence we cannot apply Theorem~\ref{thmhps}. However, the fixed point is topologically hyperbolic: for $V=\R^2$, if $B^s= [-\delta,\delta]\times\{0\}$ and $B^u=\{0\}\times[-\delta,\delta]$ for $\delta>0$, then: 
\[B=[-\delta,\delta]^2, \quad f(B)=[-\delta+\delta^3,\delta-\delta^3]\times[-2\delta,2\delta].\] 
Hence condition~(\ref{hyptop}) is easily seen to be satisfied for any $0<\delta<1$. Moreover, the coordinates axes are invariant subspaces and, with respect to the Euclidean scalar product, the cone condition~(\ref{domtop}) is plainly satisfied. Then Theorem~\ref{thmhpstop} gives the existence of at least one topologically hyperbolic fixed point in $B$, for any small $C^1$-perturbation of $f$. It is very easy to see on examples that the fixed point is in general non-unique, and moreover it can be non-isolated. Let us also add that our result not only give the existence of a topologically hyperbolic fixed point, but also the existence of local stable and unstable manifolds. 

Now a slightly less trivial example, when the submanifold is not reduced to a point, can be constructed as follows. Consider the circle diffeomorphism $b : \T \rightarrow \T$, $\T=\R/2\pi\Z$, defined by
\[ b(\theta)=\theta-\alpha\sin\theta, \quad 0<\alpha<1. \]
It has exactly two fixed points, $\theta=0$ and $\theta=\pi$ which are respectively attracting ($b'(0)=1-\alpha$) and repelling ($b'(\pi)=1+\alpha$). All other points in $\T$ are asymptotic to $0$ (respectively $\pi$) under positive (respectively negative) iterations. Consider a map $f$ which is a skew-product on $\R^2$ over $b$, of the form
\[ f : \T \times \R^2 \rightarrow \T \times \R^2, \quad f(\theta,x,y)=(b(\theta),f^s_\theta(x), f^u_\theta(y)).  \]
We assume that for some $\beta$ with $\alpha<\beta<1$,
\[ (f^s_0,f^u_0)(x,y)=((1-\beta)x,y+y^3), \quad (f^s_{\pi},f^u_{\pi})(x,y)=(x-x^3,(1+\beta)y),  \]
and we extend $f^s_\theta$ and $f^u_\theta$ for $\theta\in\T$ as follows: for every $0<\delta<1$ and $\theta\in\T$, $f^s_\theta$ and $(f^u_\theta)^{-1}$ fix $0$ and send $[-\delta,\delta]$ into $(-\delta,\delta)$ (in particular, $f$ leaves invariant the circle $\T\simeq\T\times\{(0,0)\}$), and we also ask the invariant splitting $T_{(\theta,0,0)}(\T \times \R^2)=T_\theta\T\oplus \R_x \oplus \R_y$ to be dominated. Put $V=\T \times \R^2$, $B^s=\T \times [-\delta,\delta]\times\{0\}$, $B^u=\T \times \{0\}\times[-\delta,\delta]$ and so $B=\T \times [-\delta,\delta]^2$, for some $\delta>0$. Then condition~(\ref{hyptop}) is clearly satisfied. Moreover, by the domination hypothesis, condition~(\ref{domtop}) is also satisfied with respect to the canonical Riemannian metric. Thus $\T$ is topologically normally hyperbolic but not normally hyperbolic. Therefore the invariant circle persists under small $C^1$-perturbations of $f$.

The examples we described above seem rather artificial. However, more complicated examples with a similar flavour do appear naturally in celestial mechanics (see for instance \cite{McG73} or \cite{Mos01}). We hope that our method will be useful in such situations.

\paraga To conclude, let us point out that one can easily give examples of persistent submanifolds which are not topologically normally hyperbolic. A simple example is given by the map
\[ f(x,y)=(x-x^3,y+y^3), \quad (x,y)\in \R^2, \]
which fixes the origin. Cone condition~(\ref{domtop}) is not satisfied, so our theorem cannot be applied. However, the origin is an isolated fixed point which has a non-zero index, hence by index theory (see \cite{KH97}, section $8.4$), this fixed point persists under $C^0$-perturbations. 

In general, we ask the following question:

\begin{question}
For $r\geq 1$, under which assumptions does an invariant closed $C^r$-submanifold persist under small $C^r$-perturbations ?
\end{question} 

For an isolated fixed point, index theory provides a good answer, but in general this seems quite difficult. We hope our work could be useful towards such a general answer.

\section{Proof of Theorem~\ref{thmhps} and Theorem~\ref{thmhpstop}}\label{s4}

This section is devoted to the proof of both Theorem~\ref{thmhps} and Theorem~\ref{thmhpstop}, but first we recall some useful facts. 

\paraga To our knowledge, all the previous proofs of invariant manifolds theorems use the contraction principle (or the implicit function theorem) in a suitable Banach space. Here we shall only rely on Schauder's fixed point theorem (see \cite{GD03} for a proof).

\begin{theorem}[Schauder]\label{fix}
Let $\Gamma$ be a Banach space, $K\subseteq \Gamma$ a compact convex subset and $F : K\rightarrow K$ a continuous map. Then $F$ has a fixed point.
\end{theorem}   

In fact, the hypotheses can be weakened as follows: $\Gamma$ can be replaced by a locally convex topological vector space, and $K$ needs not to be compact as long as its image by $F$ is relatively compact in $K$. However, such a greater generality will not be needed here. 

\paraga Schauder's fixed point theorem will be used to prove the existence of an invariant submanifold, but \textit{a priori} the submanifold is not continuously differentiable. For these reasons, let us say that a subset of a smooth manifold is a Lipschitz submanifold of dimension $n$ if it is locally diffeomorphic to the graph of a Lipschitz map defined on an open set of $\R^n$. 

Then, to decide whether a Lipschitz submanifold is differentiable, we will use the following notion of tangent cone. Given an arbitrary closed subset $Z \subseteq \R^m$ and $z\in Z$, the tangent cone $TC_zZ$ of $Z$ at $z$ is the set of vectors $v\in T_z\R^m$ which can be written as
\[ v=\alpha\lim_{n\rightarrow +\infty}\frac{z-z_n}{||z-z_n||}, \quad \alpha\in\R,\]
for some sequence $z_n\in Z\setminus\{z\}$ converging to $z$. This definition is clearly independent of the choice of a norm. Also, it readily extends when $\R^m$ is replaced by the $m$-dimensional manifold $M$, choosing a local chart and then checking the definition is indeed independent of the local chart used. Note that for $z\in M$, $TC_zZ$ is always a subset of $T_zM$. Of course, if $Z$ is a differentiable $n$-dimensional submanifold, then for all $z\in Z$, $TC_zZ$ is a vector subspace as it coincides with the tangent space $T_zZ$. In the case where $Z$ is a Lipschitz submanifold, then some converse statement holds true.

Indeed, let $s : \R^n \rightarrow \R^d$ be a continuous function, $Z=\mathrm{Gr}(s)\subseteq \R^{n}\times\R^d$ and $z=(x,s(x))$ with $x\in\R^n$. Under those assumptions, if $TC_zZ$ is contained in a $n$-dimensional subspace $D_z$ of $\R^n\times\R^d$ which is transverse to $\{0\}\times\R^d$, then $D_z$ is the graph of a linear map $L_{z,s} : \R^n \rightarrow \R^d$ and we can easily prove (see \cite{Fle80} for instance) that $s$ is in fact differentiable at $x$, with $T_xs=L_{z,s}$. Moreover, if $s$ is Lipschitz, then the assumption that $D_z$ is transverse to $\{0\}\times\R^d$ is automatic, and this immediately gives the following: 

\begin{lemma}\label{diff}
Let $Z$ be a Lipschitz  submanifold of dimension $n$. If for every $z\in Z$, the tangent cone $TC_zZ$ is contained in an $n$-dimensional space $D_z$, then $Z$ is a differentiable submanifold with $T_zZ=D_z$. If moreover $z\mapsto D_z$ is continuous, then $Z$ is of class $C^1$.
\end{lemma} 

\paraga \begin{proof}[Proof of Theorem~\ref{thmhpstop}]
For a clearer exposition, we will first prove the theorem when $N$ satisfies a simple geometric model and is topologically normally contracted, that is $V=N \times \R^s$. The main ideas are already present in this simple situation. In the general case, where $N$ satisfies a general geometric model and is topologically normally hyperbolic, similar arguments will enable us to construct local stable and unstable manifolds and show their persistence. Then the invariant submanifold will be their transverse intersection. The proof is divided into three steps. The first two steps show the existence and the smoothness of the invariant submanifold in the simple case. The last step is devoted to the general case.

\medskip

\noindent\emph{Step 1: existence.}

\medskip

Recall that $B^s$ denotes a compact convex neighbourhood of $0$ in $\R^s$ and $f$ is a diffeomorphism from $B=N\times B^s$ onto its image in $V=N\times \R^s$. A Riemannian metric on $V$ is given such that conditions~(\ref{hyptop}) and~(\ref{domtop}) are satisfied. As $\partial^u B$ is empty, condition~(\ref{hyptop}) is equivalent to 
\begin{equation}\label{vois}
f(B)\subseteq \mathring{B}.
\end{equation}
Indeed, an easy connectedness argument implies property~(\ref{vois}): since $f(B)\cap \partial B=\emptyset$, the set $f(B) \cap B=f(B) \cap \mathring{B}$ is both open and closed for the topology induced on $f(B)$, moreover it is non-empty (it contains $N$ which is invariant by $f$). As $f(B)$ is connected, it follows that $f(B) \cap \mathring{B}=f(B)$ and therefore $f(B)\subseteq \mathring{B}$. 

Let us note that properties~(\ref{vois}) and~(\ref{domtop}) remain true if we replace $f$ by a $C^1$-embedding $g$ which is $C^1$-close to $f$. In other words, there exists a small neighbourhood $\mathcal{B}$ of $f$ in $\mathrm{Emb}^1(B,V)$ such that for any $g\in\mathcal{B}$,
\begin{equation}\label{condg}
g(B)\subseteq \mathring{B}, \quad T_zg(C_z^s)\subseteq \mathring{C}_{g(z)}^{s}\cup\{0\}, \quad z\in B.
\end{equation}
It sounds natural to consider the set of closed $n$-dimensional $C^1$-submanifolds $\tilde{N}$, contained in $B$ and with a tangent bundle $T\tilde{N}$ contained in $C^s$. Indeed, by~(\ref{condg}), any $C^1$-diffeomorphism $g \in \mathcal{B}$ sends this set into itself and the existence of a fixed point for this action of $g$ would give the desired invariant submanifold. Nevertheless, this set lacks compactness. 

Therefore, we consider the set $\mathcal{S}$ of Lipschitz closed $n$-dimensional submanifolds $\tilde{N}$, contained in $B$ and with a tangent cone $TC\tilde{N}$ contained in $C^s$. The action of $g\in \mathcal{B}$ on this set is 
\[ \mathcal{G} : \mathcal{S} \rightarrow \mathcal{S}, \quad \mathcal{G}(\tilde{N})=g(\tilde{N}).\]      
If $\tilde{N}$ is a Lipschitz, closed submanifold of dimension $n$, then so is $g(\tilde{N})$. By the property~(\ref{condg}), the map $\mathcal{G}$ is well-defined. 

To apply Schauder's fixed point theorem, we need to exhibit a linear structure and to do so we will restrict the map $\mathcal{G}$ to a proper $\mathcal{G}$-invariant subset of $\mathcal{S}$ which, roughly speaking, consists of Lipschitz submanifolds which are graphs over $N$. 

Let $\Gamma$ be the space of continuous sections of the trivial vector bundle $N\times \mathbb R^s \rightarrow N$. Any continuous section $\sigma\in \Gamma$ is of the form $\sigma(x)=(x,s(x))$, for a continuous function $s: N \rightarrow \R^s$. Equipped with the $C^0$-norm, $\Gamma$ is a Banach space. Let us define the Lipschitz constant of a section $\sigma\in \Gamma$ at $x\in N$ by
\[ \mathrm{Lip}_x(\sigma)=\limsup_{y\rightarrow x, \; y\in N\setminus\{x\}}\frac{||s(y)-s(x)||_x}{d(y,x)}. \]
It is not hard to check that for $\sigma\in \Gamma$, $\sigma(N)$ belongs to $\mathcal{S}$ if and only if $\sigma(x)$ belongs to $B^s$ and $\mathrm{Lip}_x(\sigma)\leq 1$ for every $x \in N$. So we consider the subset 
\[ K=\{\sigma\in \Gamma \; | \; \forall x\in N, \; \sigma(x)\in B^s, \; \mathrm{Lip}_x(\sigma)\leq 1\}. \]
This subset $K$ is convex, and it is compact by the Arzel\`a-Ascoli theorem. 

Let us show that for any $\sigma \in K$,  $\mathcal{G}(\sigma(N))=\tilde{\sigma}(N)$ for some other $\tilde{\sigma}\in K$. As we already know that $\sigma(N)\in \mathcal{S}$ for $\sigma \in K$, it remains to show that $\mathcal{G}$ preserves this graph property, and to do so we will restrict $\mathcal{B}$ to a connected neighbourhood of $f$. Indeed, by the cone condition, for every $x\in N$, the plane $F_x=\{x\}\times \mathbb R^s$ is a manifold which intersects transversally any $C^1$-manifold of $\mathcal{S}$. Thus it intersects transversally $\mathcal{G}(\sigma'(N))$, for every $g\in \mathcal{B}$ and $\sigma'\in K$ of class $C^1$. By connectedness and transversality, the intersection $\mathcal{G}(\sigma'(N)) \cap F_x$ is a unique point, since it is the case for $g=f$ and $\sigma'=0$. Now for any $\sigma \in K$, we can approximate $\sigma$ by a $C^1$-section $\sigma' \in K$ in the $C^0$-topology. This implies that, for any $g\in \mathcal{B}$, $\mathcal{G}(\sigma(N))$ is close to $\mathcal{G}(\sigma'(N))$ for the Hausdorff topology, and by the cone condition, one can check that $\mathcal{G}(\sigma(N)) \cap F_x$ remains close to $\mathcal{G}(\sigma'(N)) \cap F_x$, for every $x\in N$. Since we know that the latter set is reduced to a point, $\mathcal{G}(\sigma(N)) \cap F_x$ has to be reduced to a point too. This shows that $\mathcal{G}(\sigma(N))$ is still a graph, that is $\mathcal{G}(\sigma(N))=\tilde{\sigma}(N)$ for some other $\tilde{\sigma}\in K$.    

Therefore $\mathcal{G}$ induces a map on $K$, which is obviously continuous, and as $K$ is compact and convex, by Theorem~\ref{fix} this induced map has a fixed point $\sigma_g$. Then $N_g=\sigma_g(N)$ is a $n$-dimensional Lipschitz submanifold, contained in $S$ and invariant by $g$.

\medskip

\noindent\emph{Step 2: smoothness.}

\medskip

So far we have shown the existence of a Lipschitz invariant submanifold $N_g$. Let us prove its differentiability. The cone condition implies that for any $z$ in the maximal invariant subset of $B$,
\[ D_z=\bigcap_{k\geq 0} T_{g^{-k}(z)}g^k(C^s_{g^{-k}(z)}) \subseteq C_z^s \] 
is a $n$-dimensional subspace of $T_zM$ (see \cite{New04} for instance). The  invariance of $N_g$ implies the invariance of its tangent cone under the tangent map $Tg$ and therefore
\[ TC_zN_g= \bigcap_{k\geq 0} T_{g^{-k}(z)}g^k(TC_{g^{-k}(z)}N_g). \]
By construction, $TCN_g \subseteq C^s$, and this implies $TCN_g \subseteq D$. By Lemma~\ref{diff}, the submanifold $N_g$ is differentiable, with $T_zN_g=D_z$ for $z\in N_g$.

Let us prove that $N_g$ is continuously differentiable. Given a convergent sequence $z_n \rightarrow z$ in $N_g$, we need to show that $D_{z_n}=T_{z_n}N_g$ converges to $D_z=T_zN_g$. By compactness of the Grassmannian, it is equivalent to prove that $D_z$ is the only accumulation point of $D_{z_n}$. So let $D'_z$ be an accumulation point of $D_{z_n}$. By continuity of $C_z^s$, $D'_z$ is included in $C_z^s$. For every $k\geq 0$, by continuity of $Tg^{-k}$, $T_zg^{-k} (D'_z)$ is also an accumulation point of $T_{z_n}g^{-k}(D_{z_n})=T_{g^{-k}(z_n)}N_g$ which is then included in $C_{g^{-k}(z)}^{s}$ for the same reasons. Thus $D'_z$ is included in $\bigcap_{k\geq 0} T_{g^{-k}(z)}g^k(C^s_{g^{-k}(z)})=D_z$, and since they have the same dimension, we conclude that $D'_z=D_z$. Therefore $N_g$ is continuously differentiable.    

\medskip

\noindent\emph{Step 3: general case.}

\medskip

Now let us return to the general case where $N$ is topologically normally hyperbolic, and let us work with the general geometric model 
\[ B=B^u\oplus B^s\subset V=V^u\oplus V^s\rightarrow N.\] 
A horizontal linear distribution $H$ for $\pi : V \rightarrow N$ and a Riemannian metric on $V$ are fixed, and conditions~(\ref{hyptop}) and~(\ref{domtop}) are satisfied.

We recall that topological normal hyperbolicity is an open condition on the elements involved. Therefore, for every $g$ in a neighbourhood $\mathcal{B}$ of $f$, condition~(\ref{hyptop}) gives 
\begin{equation}\label{vois2}
g(B) \cap \partial^sB=\emptyset, \quad B \cap g(\partial^u B)=\emptyset,
\end{equation} 
and the cone condition~(\ref{domtop}) gives
\begin{equation}\label{cone2}
T_zg(C_z^{s})\subseteq \mathring{C}_{g(z)}^{s}\cup\{0\}, \quad C_{g(z)}^{u}\subseteq T_zg(\mathring{C}_z^{s})\cup\{0\},
\end{equation}
for every $z\in B$. 

As before, we will only use properties~(\ref{vois2}) and~(\ref{cone2}). Let $\mathcal{S}^u$ be the set of Lipschitz, compact $(n+u)$-dimensional submanifolds $\tilde{N}^u$, contained in $B$, with a topological boundary contained in $\partial^uB$, and with a tangent cone $TC\tilde{N}^u$ contained in $C^s$. For convenience, we add the empty set to $\mathcal{S}^u$. Now for $g\in\mathcal{B}$, we put
\[ \mathcal{G}^u : \mathcal{S}^u \rightarrow \mathcal{S}^u, \quad \mathcal{G}^u(\tilde{N}^u)=g(\tilde{N}^u) \cap B.\] 
If $\mathcal{G}^u(\tilde{N}^u)$ is empty, then the map is trivially well-defined at $\tilde{N}^u$. Otherwise, $\mathcal{G}^u(\tilde{N}^u)$ is a Lipschitz and compact $(n+u)$-dimensional. By~(\ref{cone2}), if the tangent cone of $\tilde{N}^u$ is contained in  $C^s$, then the same holds true for $\mathcal{G}^u(\tilde{N}^u)$. So to show that $\mathcal{G}^u(\tilde{N}^u)$ is well-defined we only need to check the boundary condition, and this will be a consequence of~(\ref{vois2}). Note that the boundary of $\mathcal{G}^u(\tilde{N}^u)$ is the union of $g(\partial\tilde{N}^u) \cap B$ and $g(\tilde{N}^u) \cap \partial B$. Since $\partial\tilde{N}^u \subseteq \partial^uB$ and $g(\partial^u B) \cap B =\emptyset$, then $g(\partial\tilde{N}^u) \cap B$ is empty so the boundary of $\mathcal{G}^u(\tilde{N}^u)$ actually reduces to $g(\tilde{N}^u) \cap \partial B$. As $g(B) \cap \partial^sB=\emptyset$, the boundary of $\mathcal{G}^u(\tilde{N}^u)$ is included in $\partial^uB$. Therefore $\mathcal{G}^u$ is a well-defined map.      

Consider the bundle $\pi^u : V^s \oplus B^u \rightarrow B^u$, whose fibre at $\zeta\in B^u$ is the $s$-dimensional affine subspace $F_\zeta=\{v+\zeta \; | \; v\in V^s_{\pi(\zeta)}\}$ of $V^s_{\pi(\zeta)} \oplus V^u_{\pi(\zeta)}$. Let us denote by $\Gamma^u$ the associated space of continuous sections, which is a Banach space, endowed with the $C^0$-norm. Recall that the vector bundle $\pi : V^s \oplus V^u \rightarrow N$ is equipped with a linear horizontal distribution $H$, and any $x=\pi(z) \in N$ is contained in a local trivialisation $W \subseteq N$ such that
\[\phi:\pi^{-1}(W)\rightarrow W\times \R^u\times \R^s\]
is a diffeomorphism satisfying $T_z\phi(H_z)= T_xW \times \{0\}$. Any section $\sigma^u \in \Gamma^u$ satisfies, for $\zeta \in B^u$ with $\pi(\zeta)\in W$, $\phi\circ \sigma^u(\zeta)=(\zeta,s(\zeta))$ with $s : B^u \rightarrow \R^s$ a continuous function. Let us define the Lipschitz constant of an element $\sigma^u\in \Gamma^u$ at $\zeta\in B^u$ by
\[ \mathrm{Lip}_\zeta(\sigma^u)=\limsup_{\xi\rightarrow \zeta, \; \xi\in B^u\setminus\{\zeta\}}\frac{||s(\xi)-s(\zeta)||_\zeta}{d(\xi,\zeta)}. \]
It is not hard to check that for $\sigma^u\in \Gamma^u$, $\sigma^u(B^u)$ belongs to $\mathcal{S}^u$ if and only if $\sigma^u(\zeta)$ belongs to $B_\zeta=\{v+\zeta \; | \; v\in B_{\pi(\zeta)}^s\}$ and $\mathrm{Lip}_\zeta(\sigma^u)\leq 1$ for every $\zeta \in B^u$. So we consider the subset 
\[ K^u=\{\sigma^u\in \Gamma^u \; | \; \forall \zeta\in B^u, \; \sigma^u(\zeta)\in B_\zeta, \; \mathrm{Lip}_\zeta(\sigma^u)\leq 1 \}. \]
This subset is compact by the Arzel\`a-Ascoli theorem, and it is also convex (the linearity of the horizontal distribution, that we used here through the existence of a distinguished local trivialisation, is necessary to obtain the convexity). 

By the boundary condition and the transversality given by cone condition~(\ref{cone2}), the cardinality of $f(B^u)\cap F_\zeta$ does not depend on $\zeta\in \mathring B^u$. When $\zeta$ lies in $N\subset \mathring B^u$, this cardinality is at least one since $f(B^u)$ contains $N=f(N)$. For every $\zeta\in V^u$, the intersection of $F_\zeta$ with $N$ is at most one point, so by transversality, it is also the case for the intersection with a small connected neighbourhood of $N$ in $f(B^u)$. By enlarging such a neighbourhood, transversality implies that $F_\zeta\cap f(B^u)$ is at most one point, for $\zeta\in V^u$. Therefore $F_\zeta\cap f(B^u)$ is exactly one point, for $\zeta\in B^u$. This proves that $\mathcal{G}^u(\sigma^u(B^u))$ is still a graph over $B^u$, for $g=f$ and $\sigma^u=0$. Now a similar argument as before shows that $\mathcal{G}^u$ induces a continuous map on $K^u$.

Hence we can repeat the first two steps we described above, and we find that $\mathcal{G}^u$ has a fixed point $N^u_g$, which is the graph of a $C^1$-section $\sigma_g^u : B^u \rightarrow B^u \oplus B^s$. The submanifold $N^u_g$ is a local unstable manifold, and since it is a fixed point of $\mathcal{G}^u$, it is locally invariant in the sense that $g(N^u_g)\cap B=N^u_g$.

Obviously, using $C^u$ and replacing $g$ by $g^{-1}$, we can define another set $\mathcal{S}^s$ and a map $\mathcal{G}^s : \mathcal{S}^s \rightarrow \mathcal{S}^s$ which has a fixed point $N^s_g$, given by the graph of a $C^1$-section $\sigma_g^s : B^s \rightarrow  B^u \oplus B^s$. The submanifold $N^s_g$ is a local stable manifold, and it is locally invariant in the sense that $g^{-1}(N^s_g)\cap B=N^s_g$.

The transverse intersection $N_g=N^u_g \pitchfork N^s_g$ is an $n$-dimensional closed submanifold, invariant by $g$ and topologically normally hyperbolic with the same geometric model. This accomplishes the proof.
\end{proof}

\paraga \begin{proof}[Proof of Theorem~\ref{thmhps}]

The proof is divided into two steps. In the first step, we will construct a geometric model for normally hyperbolic submanifolds to show that they are topologically normally hyperbolic. Using Theorem~\ref{thmhpstop}, this will immediately give us the existence and smoothness of the invariant submanifold. We shall also notice that an arbitrarily small geometric model can be constructed, so that in addition the invariant submanifold will be $C^1$-close to the unperturbed submanifold. Then, in the second step, we will fully use the hypotheses of normal hyperbolicity to prove that the invariant submanifold is uniformly locally maximal.

\medskip

\noindent\emph{Step 1: existence and smoothness.}

\medskip

First it is enough to consider the case where $N$ is a smooth submanifold of $M$. Indeed, there exists a  $C^1$-diffeomorphism $\phi$ of $M$ such that $\phi(N)$ is a smooth submanifold of $M$ (see \cite{Hir76}, Theorem 3$.6$). The resulting metric on $M$ is then only $C^1$, but replacing it by a smooth approximation, the submanifold $\phi(N)$, which is invariant by $\phi f\phi^{-1}$, remains normally hyperbolic for this smooth metric (up to taking $\lambda$ slightly larger).

So from now $N$ is assumed to be smooth. By definition, there is a continuous, $Tf$-invariant splitting $ TN\oplus E^s \oplus E^u$ of $TM$ over $N$. The plane field $TN$ is smooth, but $E^s$ and $E^u$ are in general only continuous, so we regard smooth approximations $V^s$ and $V^u$ of them. In particular the sum $TN \oplus V^s\oplus V^u$ is direct and equal to the restriction of $TM$ to $N$. Note that $V^s$ and $V^u$ are no longer $Tf$-invariant. However, given $\gamma>0$, by taking $V^s$ and $V^u$ sufficiently close to $E^s$ and $E^u$, if we define
\[\chi_x^u=\{v=v_1+v_2\in T_xM \; | \; v_1\in T_xN\oplus V_x^s, v_2\in V_x^u,\|v_2\|_x\le \gamma\|v_1\|_x\},\]
\[\chi_x^s=\{v=v_1+v_2\in T_xM \; | \; v_1\in T_xN\oplus V_x^u, v_2\in V_x^s,\|v_2\|_x\le \gamma\|v_1\|_x\},\]  
then the following cone property is satisfied for $x\in N$:
\begin{equation}\label{champchi}
T_xf(\chi_x^s)\subset \mathring \chi_{f(x)}^s\cup \{0\}, \quad \chi_{f(x)}^u\subset T_xf(\mathring \chi_x^u)\cup \{0\}. 
\end{equation}

The plane fields $V^s$ and $V^u$ define a smooth vector bundle $V=V^s\oplus V^u\rightarrow N$ of dimension $m$ such that the fibre at $x\in N$ is the vector space $V_x^s\oplus V_x^u$. We identify the zero section of this bundle to $N$. Since $N$ is compact, by the tubular neighbourhood theorem, there exits a diffeomorphism $\Psi$ of $V$ onto an open neighbourhood $O$ of $N$ in $M$ which is the identity when restricted to the zero section.
Let $B^s\rightarrow N$ and $B^u\rightarrow N$ be two bundles whose fibres $B_x^s$ and $B_x^u$ at $x\in N$, are convex, compact neighbourhoods of $0$ in $V_x^s$ and $V_x^u$. Let $B=B^s\oplus B^u$, and define $\Psi(B)=U$ which is a compact neighbourhood of $N$ in $M$, included in $O$. We consider the Riemannian metric on $V$ obtained by pulling-back the restriction to $O$ of the Riemannian metric on $M$.
 
Let us fix a linear horizontal distribution $(H_z)_{z\in V^s\oplus V^u}$, and consider the following cone fields over $V$:
\[ C_z^u=\{v=v_1+v_2\in T_zV \; | \; v_1\in H_z\oplus V_x^s, v_2\in V_x^u,\|v_2\|_z\le \gamma\|v_1\|_z\},\]
\[ C_z^s=\{v=v_1+v_2\in T_zV \; | \; v_1\in H_z\oplus V_x^u, v_2\in V_x^s,\|v_2\|_z\le \gamma\|v_1\|_z\},\]    
for any $z\in V$ and $x=\pi(z)\in N$. Restricting $B$ if necessary, cone property~(\ref{champchi}) implies that $f'=\Psi^{-1}f\Psi$ satisfies condition~(\ref{domtop}). Indeed, by rescaling the metric, we can define exactly the same cone fields with $\gamma=1$. Furthermore, we can choose this metric such that every vector in the complement of $\chi_x^s$ (respectively $\chi_x^u$) is contracted by $T_xf$ (respectively by $T_xf^{-1}$). Then it is easy to see that for $B$ small enough, condition~(\ref{hyptop}) is also satisfied. 

Therefore a normally hyperbolic submanifold is topologically normally hyperbolic, and Theorem~\ref{thmhpstop} can be applied: there exists a neighbourhood $\mathcal{B}$ of $f'$ in $\mathrm{Emb}^1(B,V)$ such that any $g'\in\mathcal{B}$ leaves invariant and is topologically normally hyperbolic at a $C^1$-submanifold $N_{g'}$, diffeomorphic to $N$. This gives us a neighbourhood $\mathcal{U}$ of $f$ in $\mathrm{Diff}^1(M)$ such that any $g\in\mathcal{U}$ leaves invariant and is topologically normally hyperbolic at a $C^1$-submanifold $N_g$, diffeomorphic to $N$. Moreover, $N_g \subseteq U$ and $TN_g \subseteq T\Psi(C^s \cap C^u)$, and since here one has the freedom to choose $U$ and $C^s \cap C^u$ arbitrarily small (that is, $U$ and $\gamma>0$ can be taken arbitrarily small), $N_g$ is $C^1$-close to $N$. Moreover, it is easy to check that since $N$ is normally hyperbolic, then $N_g$ is not only topologically normally hyperbolic but also normally hyperbolic.

\medskip

\noindent\emph{Step 2: uniqueness.}

\medskip

Now it remains to show that $N_g=\bigcap_{k\in \Z}g^k(U)$. Recall that $N_g=N_g^u \cap N_g^s$, where $N_g^u$ and $N_g^s$ are respectively the local unstable and stable manifolds. It is enough to show that $N_g^u=\bigcap_{k\geq 0}g^k(U)$, since an analogous argument will show that $N_g^s=\bigcap_{k\leq 0}g^k(U)$ and so that $N_g=\bigcap_{k\in \Z}g^k(U)$. 

So let us prove that $N_g^u=\bigcap_{k\geq 0}g^k(U)$ which is of course equivalent to $N_{g'}^u=\bigcap_{k\geq 0}g'^k(B)$, with $g'=\Psi^{-1}g\Psi$. The inclusion $\subseteq$ is obvious. For the other one, take $z\in\bigcap_{k\geq 0}g'^k(B)$. Then, for every $k\geq 0$, $g'^{-k}(z)$ belongs to $B$. The point $z$ belongs to $F_\zeta=\{v+\zeta \; | \; v\in V^s_{\pi(\zeta)}\}$, for a $\zeta \in B^u$. The disk $F_\zeta$ has its tangent space in the complement of $\mathring C^s$. By stability of $C^s$ under $Tg'$, the disk $g'^{-1}(F_\zeta)$ has also its tangent space in the complement of $\mathring C^s$. By condition~(\ref{hyptop}), this disk does not intersect $\partial^u B$, and by~(\ref{domtop}), it comes that $F_\zeta^1= B\cap g'^{-1}(F_\zeta)$ is connected. As $N_{g'}^u$ contains $g'^{-1}(N^u_{g'}$) and $F_\zeta$ intersects $N^u_{g'}$, the submanifold $F_\zeta^1$ contains both $g'^{-1}(z)$ and a point of $N_{g'}^u$. Applying the same argument $k$ times, it comes that the preimage $F_\zeta^k$ of $F_\zeta$ by $(g_{|B}')^k$ is a disk with tangent space in the complement of $\mathring C^s$, intersecting $N^u_{g'}$ and containing $g'^{-k}(z)$. Thus the diameter of $F_\zeta^k$ is bounded by a constant $c>0$ which depends only on $\gamma$ and the diameter of $B^s_x$, $x\in N$.  

As $N_{g'}$ is normally hyperbolic, the vectors in the complement of $\mathring C^s$ are contracted by an iterate of $Tg'$. Therefore, for $B$ and $\mathcal{B}$ sufficiently small, $F^k_{\zeta}$ is contracted by $g'^k$. Letting $k$ goes to infinity, the distance between $z$ and $N_{g'}$ goes to zero and hence $z\in N_{g'}$. This ends the proof.   
\end{proof}

{\it Acknowledgments. The authors are grateful to the hospitality of IMPA. This work has been partially supported by the Balzan Research Project of J. Palis at IMPA.} 
\addcontentsline{toc}{section}{References}
\bibliographystyle{amsalpha}
\bibliography{hps3}

\end{document}